\documentclass{amsart}
\usepackage{amssymb,amsmath}
\usepackage{ytableau}
\usepackage[enableskew]{youngtab}
\usepackage{young}
\usepackage{tikz,pgf}

\newtheorem{theorem}{Theorem}

\newtheorem{lemma}{Lemma}

\newtheorem{conjecture}{Conjecture}

\begin{document}

\title{Major index over descent for pattern-avoiding permutations}
\author{William J. Keith}
\maketitle

\begin{abstract}

An open conjecture in pattern avoidance theory is that the distribution of the major index among 321-avoiding permutations is distributed unimodally.  We construct a formula for this distribution, and in the case of 2 descents prove unimodality, with unimodality for 3 through 5 descents likely being little more complicated.  The formula refines the $q$-analogue of the Frame-Robinson-Thrall hooklength formula for two-rowed partitions, and in the latter part of the paper we discuss another theorem of the same type, and further exploration toward this question.  We also give observations on the analogous behaviors for other permutation patterns of length 3. 

\end{abstract}

\section{Introduction}

In \cite{SagSav}, Bruce Sagan and Carla Savage introduced $st$-Wilf equivalence, a refinement of Wilf equivalence for pattern avoidance classes.  (Background terms are completely defined in the next section.) An outgrowth of the study of these equivalances was interest in the study of the bivariate generating functions $$\sum_{\sigma \in {\mathcal{S}_n(\tau)}} q^{\text{sta } \sigma}t^{\text{stb } \sigma}$$ for various pairs (sta, stb) of statistics.  In this article we mainly focus on an open conjecture, first made in \cite{CEKS}, concerning two of the most widely studied statistics, the major index and the descent number.  This states (Conjecture 4.2 in \cite{CEKS}): 

\begin{conjecture} For all $n$, $i$, if $\sum_{\sigma \in {\mathcal{S}_n(321)}} q^{\text{maj } \sigma} t^{\text{des } \sigma} = \sum_{i=0}^{\lfloor n/2 \rfloor} A_{n,i}(q) t^{i}$, the polynomials $A_{n,i} (q)$ are unimodal.
\end{conjecture}

In other words, among the set of 321-avoiding permutations of length $n$ with a fixed number $k$ of descents, the major index is distributed unimodally.

In this paper we make progress toward this conjecture by finding a formula for the polynomials $A_{n,i}(q)$, which appears to break this set of permutations into finer sets among which the major index is also distributed unimodally.  We can then, by analyzing the formula, show unimodality in the case of 2 descents, but a more general insight will be needed to prove the full conjecture.

Our main theorem is

\begin{theorem}\label{MainThm} The generating function for the major inverse of standard Young tableaux of shape $(n-k,k)$ with $i$ descents is $$f_{(n-k,k),i}(q) = \frac{q^{k+i^2-i}(1-q^{n-2k+1})}{1-q^i} \left[ { {k-1} \atop {i-1} } \right]_q \left[ { {n-k} \atop {i-1} } \right]_q.$$
\end{theorem}

By the Frame-Robinson-Thrall formula it holds that, whatever the distribution $f_{(n-k,k),i}$ of the major index over two-rowed tableaux with $i$ descents,

\begin{lemma}\label{FRTLemma} $$A_{n,i}(q) = \sum_{k=i}^{\lfloor n/2 \rfloor} f_{(n-k,k),i}(q) \cdot \binom{n}{k} \frac{n-2k+1}{n-k+1}.$$
\end{lemma}

Combined with the main theorem this gives a formula for $A_{n,i}(q)$.

Given the theorem and lemma, Sagan et al.'s conjecture would follow from the statement, empirically verified for $n \leq 30$, that

\begin{conjecture} For all $n$, $k$, $i$, the polynomials $f_{(n-k,k),i}(q)$ are symmetric and unimodal with central term $ni/2$.
\end{conjecture}

We can show combinatorially that 

\begin{theorem}\label{TwoThm} For all $n$, $k$, the polynomials $f_{(n-k,k),2}(q)$ are symmetric and unimodal with central term $n$.
\end{theorem}

Hence it follows that $A_{n,2}$ is unimodal.  (Unimodality of the case $A_{n,1}(q)$ is an undergraduate exercise in showing that $A_{n,1}(q) = (1+q)^n - (1+q+\dots+q^n)$.)

In Section 2 we provide definitions of all background terms and relevant theorems from the literature.  In Section 3 we prove Theorems \ref{MainThm} and \ref{TwoThm}. In Section 4 we adduce a few interesting observations on these and related polynomials, including a \emph{non}-unimodality result for 132-avoiding permutations.  In Section 5 we conclude by considering future lines of investigation, primarily the challenge of proving unimodality for general $i$, and observe the interesting link that our formula has to the Stanley hook formula for standard Young tableaux, for which our formula is an intriguing hint of a more general refinement.

\section{Definitions and background}

A sequence $(\tau_1,\dots,\tau_k)$ is \emph{order-isomorphic} to a permutation $\pi = (\pi_1, \dots, \pi_n)$ if $\tau_i > \tau_j \iff \pi_i > \pi_j$.  If no subsequence $(\sigma_{i_1},\dots,\sigma_{i_k})$, $i_1 < \dots < i_k$, of a permutation $\sigma = (\sigma_1,\dots,\sigma_n)$ is order-isomorphic to the pattern $\pi$, we say that $\sigma$ avoids $\pi$.  Denote the set of $\pi$-avoiding permutations of length $n$ by ${\mathcal{S}}_n(\pi)$.  Two patterns $\pi$ and $\tau$ are \emph{Wilf-equivalent} if $\vert {\mathcal{S}}_n(\pi)\vert $ = $ \vert {\mathcal{S}}_n(\tau) \vert $ for all $n$.  A motivating problem in the theory of pattern avoidance is the determination of Wilf equivalence classes.  For instance, all six patterns of length 3 are Wilf equivalent, but $\vert {\mathcal{S}}_5(1234) \vert \neq \vert {\mathcal{S}}_5(2413) \vert$.

In \cite{SagSav}, a refined problem of Wilf equivalence was introduced.  Consider the distribution of a permutation statistic $\text{st}:{\mathcal{S}}_n \rightarrow \mathbb{N}$ over two different sets ${\mathcal{S}}_n(\pi)$ and ${\mathcal{S}}_n(\tau)$.  If equality holds for all $n$ for the two generating function polynomials, i.e. $$\sum_{\sigma \in {\mathcal{S}}_n(\pi)} q^{\text{st } \sigma} = \sum_{\sigma \in {\mathcal{S}}_n(\tau)} q^{\text{st } \sigma}$$ for all $n$, then we say that $\pi$ and $\tau$ are \emph{st-Wilf equivalent}. 

Clearly $st$-Wilf equivalence implies Wilf equivalence but not vice versa. For an easy example, note that $Des(\sigma) = \{ i : \sigma_i > \sigma_{i+1} \}$ is the \emph{descent set} of $\sigma$, and denote by $des(\sigma) = \vert Des(\sigma) \vert$ the \emph{descent number} of a permutation.  Then the operation of reverse-complement, $$rc(\sigma) = (n+1-\sigma_k,n+1-\sigma_{k-1},\dots,n+1-\sigma_1),$$ preserves the descent number: $des(rc(\sigma)) = des(\sigma)$.  Since $rc \left( {\mathcal{S}}_n(213) \right) = {\mathcal{S}}_n(132)$, it follows that 213 and 132 are $des$-Wilf equivalent.  However, 123 and 321 are not, despite being Wilf equivalent.

The \emph{major index} of a permutation is the statistic $maj(\sigma) = \sum_{i : \sigma_i > \sigma_{i+1}} i$.  In the study of $maj$-Wilf equivalence, it is useful to break down ${\mathcal{S}}_n(\pi)$ into its subsets with given numbers of descents, occasioning study of the two-variable polynomial $$\sum_{\sigma \in {\mathcal{S}}_n(\pi)} q^{maj(\sigma)} t^{des(\sigma)}.$$

These polynomials prove to have interesting combinatorial properties in their own right, and in this paper we focus on $$\sum_{\sigma \in {\mathcal{S}}_n(321)} q^{maj(\sigma)} t^{des(\sigma)} = \sum_{i=0}^n t^i A_{n,i}(q).$$

A polynomial $f(q) = \sum_{k=a}^n b_k q^k$, $a \geq 0$, is \emph{symmetric} if $f(q) = q^{n+a} f\left( q^{-1}\right)$, in which case its coefficients form a palindromic sequence.  It is \emph{unimodal} if the sequence of its coefficients is unimodal, i.e. there is some $j$ such that $$b_a \leq b_{a+1} \leq \dots \leq b_j \geq b_{j+1} \geq \dots b_n.$$

That $A_{n,i}(q)$ is symmetric has a short proof: $rc(\sigma)$ preserves ${\mathcal{S}}_n(321)$ as well as $des(\sigma)$, and $maj(rc(\sigma)) = ni - maj(\sigma)$ for a permutation $\sigma \in {\mathcal{S}}_n$ with $i$ descents.  Hence, for every $\sigma \in {\mathcal{S}}_n(321)$ with $i$ descents and major index $c$, there is another with major index $ni-c$, and so the polynomial $A_{n,i}(q)$ forms a palindromic sequence around center $ni/2$, with minimum degree $i^2$ and maximum degree $ni-i^2$ (if $ni/2$ is not an integer, the two coefficients to either side are equal and maximal).\footnote{The same argument applies to any pattern preserved by reverse-complement.  This is the easy generalization of a response by Stanley to a MathOverflow discussion, \cite{StanleyMO}. With sporadic exceptions these seem to be the only patterns for which major index is symmetric over descent; a master's thesis studying this phenomenon for these and other statistics can be found at \cite{JTThesis}.}

Unimodality is, as stated above, a considerably harder question still open in general.  Our main result hopefully simplifies this problem by refining the sets which must be considered, to smaller sets which still exhibit unimodality.

A \emph{partition} of $n$ is a nonincreasing sequence $\lambda = (\lambda_1, \dots, \lambda_j)$ of positive integers that sums to $n$. We say such a partition has $j$ parts and write $\lambda \vdash n$, or $\vert \lambda \vert = n$.  The \emph{Ferrers diagram} of a partition is a set of unit squares in the fourth quadrant justified to the origin, in which the $i$-th row has $\lambda_i$ squares.  A \emph{standard Young tableau} of shape $\lambda$ is a filling of the Ferrers diagram of the partition $\lambda$ with the numbers 1 through $n$ in which rows and columns increase down and right.  The Ferrers diagram of the partition $(4,2,2,1)$ of 9 and a valid standard Young tableau of this shape are illustrated below.

$$\young(\hfil\hfil\hfil\hfil,\hfil\hfil::,\hfil\hfil::,\hfil::) \quad \quad \quad \young(1249,36::,57::,8:::)$$

Denote by $SYT(\lambda)$ the number of standard Young tableaux of shape $\lambda$.

The famous \emph{Robinson-Schensted correspondence} is a bijection between permutations of length $n$ and the set of pairs of standard Young tableaux on partitions of $n$ of the same shape.  The details of the bijection will not concern us, but the bijection possesses several extremely useful properties, among them:

\begin{enumerate}
\item Permutations avoiding 321 map to pairs of partitions of shape $(n)$ (only the identity permutation $12\dots n$) or $(n-k,k)$.
\item A place $i \in Des(\sigma)$ if and only if $i$ is in a row strictly higher than $i+1$ in the second tableau of the SYT pair (the ``recording tableau'').
\end{enumerate}

The latter property defines descents and major index on tableaux completely analogously to permutations, with the descent set of a tableau being those $i$ in a strictly higher row than $i+1$ and the major index being the sum of such $i$.

The \emph{hook-length} of each square in the Ferrers diagram of $\lambda$ equals the total of the number of squares directly below and to that square's right, plus 1 for itself.  The hook lengths of $\lambda=(4,2,2,1)$ are illustrated below.

$$\young(7521,42::,31::,1:::)$$

The \emph{Frame-Robinson-Thrall hook-length formula} says that the number of standard Young tableaux of shape $\lambda$ is $$\frac{n!}{\prod h_{ij}},$$ where $h_{ij}$ is the hook length of the square with lower right corner at $(-i,-j)$ in the plane.

This is particularly useful to us since this means that, given a particular standard Young tableau of shape $(n-k,k)$ with $i$ descents and major index $M$, there are exactly $\binom{n}{k}\frac{n-2k+1}{n-k+1}$ permutations with this SYT as their recording tableau and hence this descent pattern and major index.  Thus, if we denote the polynomials $$f_{(n-k,k),i}(q) = \sum_{{T \in SYT(\lambda)} \atop {\vert Des(T)\vert = i}} q^{maj(T)},$$ it immediately follows from the above discussion that $$A_{n,i}(q) = \sum_{k=i}^{\lfloor n/2 \rfloor} f_{(n-k,k),i}(q) \cdot \binom{n}{k} \frac{n-2k+1}{n-k+1}$$ and Lemma \ref{FRTLemma} holds.

Finally, the generating function for partitions in which the largest part is at most $M$ and the number of parts is at most $N$ is the $q$-\emph{binomial coefficient} $$\sum_{{\lambda_1 \leq M} \atop {\lambda \text{ has } \leq N \text{ parts}}} q^{\vert \lambda \vert} = \left[ {{M+N} \atop N} \right]_q =: \frac{(1-q)(1-q^2)\dots(1-q^{M+N})}{(1-q)\dots(1-q^M)(1-q)\dots(1-q^N)} =: \frac{(q)_{M+N}}{(q)_M (q)_N}.$$ Here $(q)_N = (1-q)\dots(1-q^N)$. We have the special cases $\left[ {j \atop 0} \right]_q = 1$ for $j \geq 0$, $\left[ {j \atop k} \right]_q = 0$ for $k \geq 0$, $j < k$.

\section{Proof of the main theorems}

\subsection{Proof of Theorem \ref{MainThm}}

We prove the theorem by recurrence, beginning with the base case $i=1$.

A SYT of shape $(n-k,k)$ with exactly 1 descent must be of the form

\begin{center}
\begin{tabular}{c}
\begin{tikzpicture}
\draw (-0.2,0) -- (7.5,0) -- (7.5,-0.5) -- (-0.2,-0.5) -- (-0.2,-1) -- (3,-1);
\draw (-0.2,0) -- (-0.2,-0.5);
\draw (0.5,0) -- (0.5,-1);
\draw (1,0) -- (1,-1);
\draw (1.35,0) -- (1.35,-1);
\draw (1.85,0) -- (1.85,-1);
\draw (2.3,0) -- (2.3,-1);
\draw (3,0) -- (3,-1);
\draw (3.75,0) -- (3.75,-0.5);
\draw (4.5,0) -- (4.5,-0.5);
\draw (5,0) -- (5,-0.5);
\draw (6.25,0) -- (6.25,-0.5);
\draw(6.75,0) -- (6.75,-0.5);
\draw (0.1,-0.25) node {1};
\draw (0.7,-0.25) node {2};
\draw (1.2,-0.25) node {3};
\draw (1.65,-0.25) node {\dots};
\draw (2.1,-0.25) node {k-1};
\draw (2.7,-0.25) node {k};
\draw (3.35,-0.25) node {k+1};
\draw (4.2,-0.25) node {\dots};
\draw (4.8,-0.25) node {j};
\draw (5.6,-0.25) node {j+k+1};
\draw (6.5,-0.25) node {\dots};
\draw (7,-0.25) node {n};
\draw (0.2,-0.75) node {j+1};
\draw (0.8,-0.75) node {\dots};
\draw (2.65,-0.75) node {j+k};
\end{tikzpicture}
\end{tabular}
\end{center}

\noindent for some $k \leq j \leq n-k$.  Its major index is precisely $j$. Hence $$f_{(n-k,k),1}(q) = q^k+\dots + q^{n-k}.$$

This is the claim of the theorem for the case $i=1$.

We observe the recurrence

\begin{lemma} For $i>1$, $$f_{(n-k,k),i}(q) = f_{(n-k-1,k),i}(q) + \sum_{k_0=i-1}^{k-1} q^{n-k+k_0}f_{(n-k-1,k_0),i-1}(q).$$
\end{lemma}

\noindent (We only define SYT on partitions and so there are no SYT of ``shape'' $(k-1,k)$.  Hence we set $f_{(k-1,k),i}(q) = 0$ for all $i$ and all $k>0$.)

\begin{proof}

Either the top row of the SYT ends with $n$, or it does not.  If it does, then removal of the box containing $n$ results in a SYT with $i$ descents of shape $(n-k-1,k)$, and this is a bijection since by taking any such tableau and appending $n$ to the top row we obtain a SYT of the foregoing type.  This is the first term.

The second term arises when the final element of the top row is $n-k+k_0$, $k_0 < k$.  Then this element is the top of a descent, and the tableau is of the form

\begin{center}
\begin{tabular}{c}
\begin{tikzpicture}
\draw (0,0) -- (9.6,0) -- (9.6,-0.5) -- (0,-0.5) -- (0,-1) -- (5.7,-1) -- (5.7,0);
\draw (0,0) -- (0,-0.5);
\draw (8.4,-0.25) node {$n-k+k_0$};
\draw (3.1,-0.75) node {$n-k+k_0+1$};
\draw (1.8,0) -- (1.8,-1);
\draw (0.6,0) -- (0.6,-1);
\draw (1.2,0) -- (1.2,-1);
\draw (0.3,-0.25) node {$\blacksquare$};
\draw (0.3,-0.75) node {$\blacksquare$};
\draw (0.9,-0.25) node {$\dots$};
\draw (0.9,-0.75) node {$\dots$};
\draw (1.5,-0.25) node {$\blacksquare$};
\draw (1.5,-0.75) node {$\blacksquare$};
\draw (4.5,0) -- (4.5,-1);
\draw (3.1,-0.25) node {$\blacksquare$};
\draw (5.1,0) -- (5.1,-1);
\draw (4.8,-0.25) node {$\dots$};
\draw (4.8,-0.75) node {$\dots$};
\draw (5.4,-0.25) node {$\blacksquare$};
\draw (5.4,-0.75) node {$n$};
\draw (6,-0.25) node {$\blacksquare$};
\draw (6.3,0) -- (6.3,-0.5);
\draw (6.6,-0.25) node {$\dots$};
\draw (6.9,0) -- (6.9,-0.5);
\draw (7.5,0) -- (7.5,-0.5);
\draw (7.2,-0.25) node {$\blacksquare$};
\end{tikzpicture}
\end{tabular}
\end{center}

\noindent where the black squares constitute a standard Young tableau in which the top row is of length $n-k-1$, the bottom row is of length $k_0 < k$, and there are exactly $i-1$ descents.  Once $n$, $k$, and such a tableau are selected, there is only one way to create a tableau by appending a single element $n-k+k_0$ to the end of the first row, and filling the remainder of the bottom row to length $k$ with the remaining elements.  This adds $n-k+k_0$ to the major index of such a tableau.  The possible values for $k_0$ run from a minimum of $i-1$ to a maximum of $k-1$.  These statements give the second term, and thus the recurrence holds.

\end{proof}

Observe that for the partition $(k,k)$, i.e. $k = \frac{n}{2}$, there are no partitions of shape $(k-1,k)$ and, likewise, that there are no terms that arise by removing only the last element of the first row, because this element must be the top of a descent.

Therefore, for $i>1$, it suffices for the induction to assume that we have established the truth of the formula for $i-1$ for all $n$, $k$, and for $(n-1,k)$ when this exists.  We have established the truth of the formula for $i=1$ for all $n$ and $k$ and now proceed to sum the recurrence.

We require two summation identities on $q$-binomial coefficients which we state in the next lemma.

\begin{lemma} For $m,n \geq 0$, 
\begin{align}
\sum_{j=0}^n q^j \left[ {{m+j} \atop m} \right]_q &= \left[ {{n+m+1} \atop {m+1}} \right]_q \\
\sum_{j=a}^B q^{2j} \left[ {j \atop a} \right]_q &= q^{2a}  \left( \left[ {{B+2} \atop {a+2}} \right]_q - q \left[ {{B+1} \atop {a+2}} \right]_q \right) .
\end{align}
\end{lemma}

The first identity is standard and may be found as (\cite{ToP}, (3.3.9)).  The second identity is almost certainly somewhere in the literature but, as it was not found in several standard textbooks, a short proof is here adduced.

\begin{proof}

A term $q^{2j+4} \left[ {j \atop a} \right]_q$ can be interpreted as the generating function for partitions into parts of size at most $a+2$, with at most $j-a+2$ parts, under the additional stipulation that two parts of size $a+2$ exist and all other parts are of size at least 2.  Visually, these are partitions in a bounding box of size $(j-a+2) \times (a+2)$ with a guaranteed ``fat hook'':

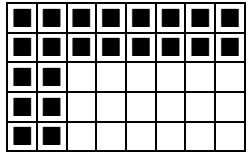
$$ \young(\blacksquare\blacksquare\blacksquare\blacksquare\blacksquare\blacksquare\blacksquare\blacksquare,\blacksquare\blacksquare\blacksquare\blacksquare\blacksquare\blacksquare\blacksquare\blacksquare,\blacksquare\blacksquare\hfil\hfil\hfil\hfil\hfil\hfil,\blacksquare\blacksquare\hfil\hfil\hfil\hfil\hfil\hfil,\blacksquare\blacksquare\hfil\hfil\hfil\hfil\hfil\hfil)$$

Here the black squares must be part of the partition, and the remaining squares may be occupied with any partition in the $6 \times 3$ box.

Rewrite the identity as 

$$\sum_{j=a}^B q^{2j} \left[ {j \atop a} \right]_q = q^{-4} \sum_{j=a}^B q^{2j+4} \left[ {j \atop a} \right]_q.$$

If we remove the two required largest parts of size $a+2$, we are left with partitions in a $(j-a) \times (a+2)$ box in which all parts must be of size at least 2:

$$ \young(\blacksquare\blacksquare\hfil\hfil\hfil\hfil\hfil\hfil,\blacksquare\blacksquare\hfil\hfil\hfil\hfil\hfil\hfil,\blacksquare\blacksquare\hfil\hfil\hfil\hfil\hfil\hfil).$$

As $j$ ranges from $a$ to $B$, the height of the box ranges from 0 to $B-a$.

Consider all partitions in the $(B-a) \times (a+2)$ box, which are counted by the generating function $\left[ {{B+2} \atop {a+2}} \right]_q$.  Such a partition is counted by our sum exactly once if its smallest part is at least 2; in this case the index $j$ at which the partition appears is the $j$ at which $j-a$ is the number of parts in the partition.

Partitions in the $(B-a) \times (a+2)$ box are not counted by our sum if they have a part of size 1.  We can count all such partitions by taking any partition in the $(B-a-1) \times (a+2)$ box and appending a 1; their generating function is $q \left[ {{B+1} \atop {a+2}} \right]_q$.

Hence the partitions counted by the sum are counted precisely by the difference between the two generating functions:

\begin{multline*}\sum_{j=a}^B q^{2j} \left[ {j \atop a} \right]_q = q^{-4} \sum_{j=a}^B q^{2j+4} \left[ {j \atop a} \right]_q = q^{-4} \left( q^{2a+4} \right) \left( \left[ {{B+2} \atop {a+2}} \right]_q - q \left[ {{B+1} \atop {a+2}} \right]_q \right) \\  = q^{2a} \left( \left[ {{B+2} \atop {a+2}} \right]_q - q \left[ {{B+1} \atop {a+2}} \right]_q \right)
\end{multline*}

\noindent and the lemma is proved.

\end{proof}

Proof of Theorem \ref{MainThm} is now simply expanding the recurrence, employing the summation identities above, and verifying the resulting equality.

\begin{align*}
f_{(n-k,k),i}(q) &= f_{(n-k-1,k),i}(q) + \sum_{k_0=i-1}^{k-1} q^{n-k+k_0}f_{(n-k-1,k_0),i-1}(q) \\
&= \frac{q^{k+i^2-i}(1-q^{n-2k})}{1-q^i} \left[ {{k-1} \atop {i-1}} \right]_q \left[ {{n-k-1} \atop {i-1}} \right]_q \\ &+ \sum_{k_0 = i-1}^{k-1} q^{n-k+k_0} \frac{q^{k_0+i^2-3i+2}(1-q^{n-k-k_0})}{1-q^{i-1}} \left[ {{k_0-1} \atop {i-2}} \right]_q \left[ {{n-k-1} \atop {i-2}} \right]_q \\
&= \frac{q^{k+i^2-i}(1-q^{n-2k})}{1-q^i} \left[ {{k-1} \atop {i-1}} \right]_q \left[ {{n-k-1} \atop {i-1}} \right]_q \\ 
&+ \frac{q^{n-k+i^2-3i+2}}{1-q^{i-1}} \left[ {{n-k-1} \atop {i-2}} \right]_q \sum_{k_0=i-1}^{k-1} \left( q^{2k_0} - q^{n-k+k_0} \right) \left[ {{k_0-1} \atop {i-2}} \right]_q \\
&= \frac{q^{k+i^2-i}(1-q^{n-2k})}{1-q^i} \left[ {{k-1} \atop {i-1}} \right]_q \left[ {{n-k-1} \atop {i-1}} \right]_q \\
&+ \frac{q^{n-k+i^2-3i+2}}{1-q^{i-1}} \left[ {{n-k-1} \atop {i-2}} \right]_q \left( q^2 \left( \sum_{j_1 = i-2}^{k-2} q^{2j_1} \left[ {{j_1} \atop {i-2}} \right]_q \right) \right. \\
&+ \left. \left( q^{n-k+i}  \sum_{j_2 = 0}^{k-i} q^{j_2} \left[ {{(i-2)+j_2} \atop {i-2}} \right]_q \right) \right)
\end{align*}

In the last equality we made the substitutions $j_1 = k_0 - 1$ and $j_2 = k_0 - i + 1$.

Now employing the $q$-binomial summation identities of the previous lemma, we obtain

\begin{align*}
f_{(n-k,k),i}(q) &= \frac{q^{k+i^2-i}(1-q^{n-2k})}{1-q^i} \left[ {{k-1} \atop {i-1}} \right]_q \left[ {{n-k-1} \atop {i-1}} \right]_q \\
&+ \frac{q^{n-k+i^2-3i+2}}{1-q^{i-1}} \left[ {{n-k-1} \atop {i-2}} \right]_q \left( q^2 \left( q^{2i-4} \left[ {k \atop i} \right]_q - q^{2i-3} \left[ {{k-1} \atop i} \right]_q \right) \right. \\ 
&- \left. q^{n-k+i-1} \left[ {{k-1} \atop {i-1}} \right]_q \right).
\end{align*}

Set this final sum conjecturally equal to $\frac{q^{k+i^2-i}(1-q^{n-2k+1})}{1-q^i} \left[ {{k-1} \atop {i-1}} \right]_q \left[ {{n-k} \atop {i-1}} \right]_q$ and multiply through the claimed equality by $$ \frac{(q)_i (q)_{i-1} (q)_{k-i+1} (q)_{n-k-i+1}}{q^{k+i^2-i} (q)_{k-1} (q)_{n-k-1}} .$$ We obtain that the conjecture is equivalent to the truth of the equation

\begin{multline*}(1-q^{k-i+1})(1-q^{n-k})(1-q^{n-2k+1}) = \\
(1-q^{n-2k})(1-q^{k-i+1})(1-q^{n-k-i+1}) + q^{n-2k} (1-q^{k-i+1})(1-q^k) \\ 
- q^{n-2k+1}(1-q^{k-i})(1-q^{k-i+1}) - q^{2n-3k-i+1}(1-q^i)(1-q^{k-i+1}).
\end{multline*}

Expansion and cancellation verifies the identity, and the theorem is proved.

\hfill $\Box$

\subsection{Proof of Theorem \ref{TwoThm}}

Having observed that unimodality and equicentrism of the polynomials $f_{(n-k,k),i} (q)$ implies unimodality of $A_{n,i}(q)$, in this section we show these facts for $f_{(n-k,k),2}(q)$.  

The formula in this case simplifies to

$$f_{(n-k,k),2}(q) = q^{k+2}\frac{1}{1+q}(1+q+\dots+q^{k-2})(1+\dots+q^{n-k-1})(1+\dots+q^{n-2k}).$$

If $j$ is odd, then $\frac{1}{1+q}(1+q+\dots+q^j) = 1+q^2+q^4+\dots+q^{j-1}$.  We have three cases: $k$ is odd, $k$ and $n$ are both even and thus the second factor has odd final exponent, or $k$ is even and $n$ is odd and the third final exponent is odd.

If $k$ is odd, we wish to establish the unimodality of the product

\begin{multline*}(1+q^2+q^4+\dots+q^{k-3})(1+\dots+q^{n-k-1})(1+\dots+q^{n-2k}) \\
= (1+q^2+q^4+\dots+q^{k-3})\left(1+2q+3q^2+\dots+(n-2k+1)q^{n-2k}+(n-2k+1)q^{n-2k+1} \right. \\
\left. + \dots + (n-2k+1)q^{n-k-1} + (n-2k)q^{n-k} + \dots +3q^{2n-3k-3}+2q^{2n-3k-2} + q^{2n-3k-1}\right).
\end{multline*}

The coefficients of the second factor increase from the $q^0$ term to the $q^{n-2k}$ term, stay constant until the $q^{n-k-1}$ term, and thereafter decrease to the $q^{2n-3k-1}$ term.  The whole product consists of overlapping copies of this latter factor.  But between $q^{n-2k}$ and $q^{n-k-1}$ there are $k$ terms inclusive, and thus the falling portions of shifting copies of the second factor never begin before the rising portions are done.  Thus, the product is unimodal.  (In fact we could give it explicitly if we wished.)

In all of these cases, we wish to establish the unimodality of a product 
\begin{multline*}(1+q^2+q^4+\dots+q^{j-1}) \\ (1+2q+3q^2+\dots+mq^a+mq^{a+1}+\dots+mq^b+(m-1)q^{b+1}+\dots+q^c).
\end{multline*}

While for other values of $n$ and $k$ it is no longer necessarily case that the rising portions of copies of the latter factor terminate before any falling portions begin, it is clearly the case that in the sequence of coefficients there is first a segment where there are a number of rising copies greater than the number of falling copies (and hence the sequence is increasing), followed by a segment in which the number of rising and falling copies are equal, followed by a segment in which the number of falling copies is greater than the number of rising copies.  The sequences of coefficients rise and fall at a rate of precisely 1 per step, and so knowing the numbers of rising and falling copies gives the rate of increase or decrease of the sum.

Thus, the sequence of coefficients in the product is unimodal, and from the arguments above it can also be seen that they are symmetric.  In any of the three cases, the polynomial has minimum degree $k+2$ and maximum degree $2n-k-2$, and hence its peak terms occur at (and symmetrically around) the term $q^{n}$. Their sum is thus likewise a symmetric, unimodal polynomial around this peak, and Theorem \ref{TwoThm} is proved. 

\hfill $\Box$

\section{Observations on related polynomials}

In this section we include a few remarks on the polynomials related to other patterns.

For the remainder of this section denote by $$F_{\tau,n} = \sum_{\sigma \in {\mathcal{S}}_n(\tau)} q^{maj \, \sigma} t^{des \, \sigma} =: \sum_{i=0}^{n-1} t^i g_{\tau,n,i} (q).$$

\subsection{123- and 321-avoiding permutations}

The reversal $r(\sigma)$ of a 321-avoiding permutation $\sigma$ of length $n$ with $i$ descents and $maj(\sigma) = M$ is a 123-avoiding permutation with $n-1-i$ descents, $maj(r(\sigma)) = n(n-1)/2 - ni + M$.  Hence $$g_{123,n,i}(q) = A_{n,n-1-i}(q) \cdot q^{n(n-1)/2-ni}$$ and the $g_{123,n,i}(q)$ are thus also symmetric and (where proven, i.e. at least for $i=0,1,2$ so far) unimodal.  Indeed, since for $n$ odd the maximum number of descents in a 321-avoiding permutation of length $n$ is $(n-1)/2$, which is also the minimum number of descents in a 123-avoiding permutation of length $n$, the $q$-power factor is $q^0$ and the two relevant terms are exactly the same.

An example of these behaviors for $n=5$ is below.

\begin{align*}
F_{321,5} &= 1 + (4q+9q^2+9q^3+4q^4)t + (5q^4+5q^5+5q^6)t^2 \\
F_{123,5} &= (5q^4+5q^5+5q^6)t^2+(4q^6+9q^7+9q^8+4q^9)t^3 + q^{10} t^4
\end{align*}

One notes that the last coefficient in $F_{321,5}$ is not only unimodal but in fact constant in its coefficients.  Let $C_j = \frac{1}{j+1}\binom{2j}{j}$ be the $j$-th Catalan number. By the Robinson-Schensted correspondence we have the following theorem.

\begin{theorem} $$g_{321,n,\lfloor n/2 \rfloor} = \left\{ \begin{array}{lr} C_{\lfloor n/2 \rfloor} \cdot q^{\lfloor n/2 \rfloor^2 + \lfloor n/2 \rfloor} & $n$ \, \text{ even} \\ \frac{4 \lfloor n/2 \rfloor + 2}{\lfloor n/2 \rfloor + 2} C_{\lfloor n/2 \rfloor} \cdot \left( q^{\lfloor n/2 \rfloor^2} + \dots + q^{\lfloor n/2 \rfloor^2 + \lfloor n/2 \rfloor}\right) & $n$ \, \text{ odd.} \end{array} \right.$$
\end{theorem}

\begin{proof} If $n$ is even, the only possible recording tableau with $n/2$ descents is $$\young(135{{\dots}},246{{\dots}}).$$  The number of SYT in the $2 \times j$ rectangle, calculable from the Frame-Robinson-Thrall formula, is well-known to be $C_j$. (Indeed the fact that 321-avoiding alternating permutations of even length are counted by the Catalan numbers is problem 146 in Stanley's famous Catalan number collection \cite{catadd}.) 

If $n$ is odd, the only possible recording tableaux consist of varying $$\young(1357{\,{\dots}\,}n,246{{\dots}}{n-1\,}:)$$ by proceeding one step further in the top row before any one descent:

$$\young(12468{{\dots}},357{{\dots}}{n}:) \quad , \quad \young(13468{{\dots}},257{{\dots}}{n}:) \quad , \quad \young(13568{{\dots}},247{{\dots}}{n}:) , \quad \dots$$

The major indices of these tableaux range from $\lfloor n/2 \rfloor^2$ to $\lfloor n/2 \rfloor^2 + \lfloor n/2 \rfloor$, and each has the same number of permutations associated with it as recording tableau; Frame-Robinson Thrall again gives the claimed constant.
\end{proof}

\subsection{132-, 231-, 312-, and 213-avoiding permutations}

The patterns 132, 231, 312, and 213 are related by reverse, complement, and reverse-complement. The effect of reverse was described above. Complementation sends a permutation of length $n$ with $i$ descents and major index $M$ to one with $n-1-i$ descents and major index $\binom{n}{2} - M$.  

Let $f(q)$ be a polynomial of degree at most $J$.  Denote the \emph{reversal} of $f(q) = a_0 + \dots + a_Jq^J$ within the interval $\left[ 0, J \right]$ by $rev_J(f(q)) = a_J + \dots + a_0 q^J$. From the foregoing facts we have the following relationships.

\begin{align*} F_{132,n} &= g_{132,n,0} + g_{132,n,1}t + \dots + g_{132,n,n-1}t^{n-1} \\
F_{231,n} &= g_{132,n,n-1}q^{\binom{n}{2} -n(n-1)} + g_{132,n,n-2}q^{\binom{n}{2}-n(n-2)}t + \dots + g_{132,n,0}q^{\binom{n}{2} -n(0)}t^{n-1} \\
F_{213,n} &= rev_{\binom{n}{2}} (g_{132,n,n-1}) + rev_{\binom{n}{2}} (g_{132,n,n-2}) t + \dots rev_{\binom{n}{2}} (q_{132,n,0}) t^{n-1} \\
F_{312,n} &= rev_{0}(g_{132,n,0}) + rev_{n}(g_{132,n,1}) t + \dots + rev_{n(n-1)}(g_{132,n,n-1})t^{n-1}
\end{align*}

A standard bijection (recursively switch the values of the portions of the permutation before and after element $n$) associates to any 132-avoiding permutation a 231-avoiding permutation with the same descent set, and hence from the first two lines of the previous equations we have $$g_{132,n,i} = g_{132,n,n-1-i}q^{\binom{n}{2}-n(n-1-i)}.$$ Thus, the $q$-polynomial coefficients of $t^i$ and $t^{n-1-i}$ in $F_{132}$ are almost equal, except for shift by a power of $q$.  

Hence we have in fact $F_{132,n} = F_{231,n}$, $F_{213,n} = F_{312,n}$, and $F_{213,n}$ is the same as $F_{132,n}$ except that each $g_{132,n,i}$ has its sequence of coefficients reversed.  The examples of this behavior for $n=5$ are exhibited below.

\begin{align*}
F_{132,5} &= 1+(4q+3q^2+2q^3+q^4)t+(6q^3+5q^4+6q^5+2q^6+q^7)t^2 \\ &+(4q^6+3q^7+2q^8+q^9)t^3 + q^{10}t^4 \\
F_{231,5} &= 1+(4q+3q^2+2q^3+q^4)t+(6q^3+5q^4+6q^5+2q^6+q^7)t^2 \\ &+(4q^6+3q^7+2q^8+q^9)t^3 + q^{10}t^4 \\
F_{213,5} &= 1+(q+2q^2+3q^3+4q^4)t+(q^3+2q^4+6q^5+5q^6+6q^7)t^2 \\ &+(q^6+2q^7+3q^8+4q^9)t^3 + q^{10}t^4 \\
F_{312,5} &= 1+(q+2q^2+3q^3+4q^4)t+(q^3+2q^4+6q^5+5q^6+6q^7)t^2 \\ &+(q^6+2q^7+3q^8+4q^9)t^3 + q^{10}t^4 \\
\end{align*}

Thus, we may learn about any of these polynomials by focusing on one, and hence for the remainder of this section we will consider 132-avoiding polynomials.  For these, we have the opposite of our results for 321: it is \emph{always} the case that we have non-symmetry and non-unimodality for at least some of the $g_{132,n,i}$, which we formalize below.

\begin{theorem} For $n\geq 3$, $g_{132,n,1}$ is always non-symmetric, and for $n \geq 5$, $g_{132,n,2}$ is always non-unimodal, due to the coefficients being given by the formula: \begin{multline*}F_{132,n} = 1 + \left( \sum_{i=1}^{n-1} (n-i) q^i \right) t \\ + \left( \binom{n-1}{2} q^3 + \left( \binom{n-1}{2} - 1 \right) q^4 + (n^2-4n+1) q^5 + \dots \right) t^2 + \dots .\end{multline*}
\end{theorem}

\begin{proof}

For one descent, fix a place $i$ for the only descent; the entries before and after $i$ must ascend. Once we select the values of the entries prior to $i$, the permutation is determined, since these values in ascending order are placed up to the $i$-th place and the remaining values are placed in ascending order afterward.

$$45 \vert 123 \,\,\, 6789$$

In order that the permutation contain no 132 pattern, no entry in the latter portion of the permutation can appear between two values chosen in the former portion; thus, the values prior to place $i$ must be an interval of length $i$.  There are $n-i$ such intervals that can be chosen (the latter portion must contain 1).

For two descents, we take each major index separately.

\noindent \textbf{Case 1:} $maj(\sigma) = 3$.  Descents must occur at places 1 and 2.  Choose 2 values from $\{2,\dots,n\}$ and place them in descending order in places 1 and 2.  Place the remaining values in ascending order.  There are $\binom{n-1}{2}$ such permutations.

\noindent \textbf{Case 2:} $maj(\sigma) = 4$.  Descents must occur at places 1 and 3.  Choose values $(a,b)$, $a > b$, for places 1 and 2 from among $\{2, \dots , n\}$, with the exception of the pair $(n,n-1)$.  Place them in descending order.  The entry in place 3 must be $b+1$, or $b+2$ if $a=b+1$, in order to avoid a 132 pattern.  There are $\binom{n-1}{2} - 1$ such permutations.

\noindent \textbf{Case 3:} $maj(\sigma) = 5$.  Either descents occur at places 1 and 4, or they occur at places 2 and 3.  

For descents at places 1 and 4, one of two schemes is used.  We may choose $(a,b)$, $a > b$, for the first two places from among $\{2, \dots , n-3\}$. We require places 3 and 4 to be $b+1$ and $b+2$, unless $a$ is already one of these values, in which case we need $b+3$.  The remaining entries are arranged in ascending order.  There are $\binom{n-4}{2}$ such permutations.  We can additionally choose $a$ from among $\{n,n-1,n-2\}$, though not $b$.  There are $3n-12$ such permutations.

For descents at places 2 and 3, we simply need to choose values $(a,b)$, $a > b$, for places 1 and 3 from among $\{2,\dots,n-1\}$ and arrange these in descending order; the value at place 2 must be $a+1$.  There are $\binom{n-2}{2}$ such permutations.

Totaled, there are then $\binom{n-4}{2} + 3n-12 + \binom{n-2}{2} = n^2-4n+1$ permutations avoiding 132 with major index 5, as claimed.

Since $n^2-4n+1 > \binom{n-1}{2} - 1$ for $n \geq 3$, the theorem is proved.

\end{proof}

\section{Future directions}

\subsection{The main conjecture}

The main conjecture of Cheng et al. remains open.  It is easy to observe that for all $i$, the polynomials $f_{(n-k,k),i}(q)$ are symmetric: consider the form $$f_{(n-k,k),i} = \frac{q^{k+i^2-i}(1+q+q^2+\dots+q^{n-2k})}{1+q+q^2+\dots+q^{i-1}} \left[ { {k-1} \atop {i-1} } \right]_q \left[ { {n-k} \atop {i-1} } \right]_q.$$ The numerator (the expression other than the term $1+q+\dots+q^{i-1}$) is the product of symmetric polynomials and hence symmetric, and since we know the quotient is also polynomial, the numerator is divisible by $1+q+\dots+q^{i-1}$.  The quotient of a symmetric polynomial by $1+q+\dots+q^{i-1}$ is also symmetric (for contradiction, if the quotient were not symmetric, take the outermost pair of nonequal coefficients in symmetric positions and observe the value of the coefficients of the product at the same distance from the ends).

The $f_{(n-k,k),i}(q)$ likewise have mean degree $ni/2$, as they have minimal degree $k+i^2-i$ and maximal degree $ni-k-i^2+i$: the largest term of $\left[ {{M+N} \atop N}\right]_q$ is $q^{MN}$, counting the partition filling the box.

Thus to prove the full conjecture it remains to determine the unimodality of the expressions individually.  The $q$-binomial coefficients are unimodal; the product of symmetric unimodal polynomials is unimodal; the final step is an exploration of the effect of division by $1+q+\dots+q^{i-1}$.

An approach not taken in this article could consider a different route.  The distribution of the major index of two-rowed standard Young tableaux of shape $(n-k,k)$ with a fixed number of descents is equivalent to the distribution of the sums of place values of peaks in lattice paths that stay above the main diagonal in the $(n-k) \times k$ box:

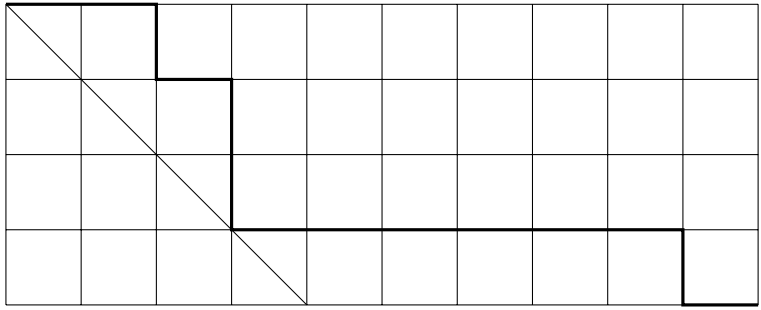
\begin{center}
\begin{tabular}{c}
\begin{tikzpicture}
\draw (0,0) grid (10,-4);
\draw (0,0) -- (4,-4);
\draw[very thick] (0,0) -- (2,0) -- (2,-1) -- (3,-1) -- (3,-3) -- (9,-3) -- (9,-4) -- (10,-4);
\end{tikzpicture}
\end{tabular}
\end{center}

This path would correspond to the SYT $$\young(124789{{10}}{{11}}{{12}}{{14}},356{{13}}).$$ Peaks occur after steps 2, 4, and 12, the sum of which is the major index.  Perhaps lattice paths in a given size of bounding grid with a fixed number of peaks are susceptible to the construction of a symmetric chain decomposition, a standard combinatorial tool for proving unimodality.

\noindent \textbf{Remark:} The $q$-binomial coefficients $\left[ {{M+j} \atop j} \right]_q$ have relatively simple descriptions for $j=2,3,4$, since there are known symmetric chain decompositions for these generating functions.  By analysis of explicit forms of the products involved, a diligent student could probably prove unimodality for $A_{n,i}(q)$ for $i=3,4,5$.

\subsection{The Stanley hook formula}

If $\lambda \vdash n$ and we ignore the number of descents, the distribution of the major index over $SYT(\lambda)$  is known (\cite{EC2}, Corollary 7.21.5): 

\begin{theorem} Denote $b(\lambda) = \sum (i-1)\lambda_i$.  Then $$\sum_{T \in SYT(\lambda)} q^{maj(T)} = \frac{q^{b(\lambda)} (q)_n}{\prod (1-q^{h_{ij}})}.$$
\end{theorem}

Observe that this is a $q$-analogue of the Frame-Robinson-Thrall formula given earlier.  Our formula then refines this formula by descent number for the case of two-rowed partitions.  We note that the binomial formula for which the main formula of this paper is a $q$-analogue can itself be found in the literature \cite{Barahovski}.

A natural question is whether similar formulae can be found for partitions of other shapes.  This could be motivated by pattern avoidance questions, in which we would be interested in the question of whether the distribution of the major index of tableaux with given shapes and specified numbers of descents is itself unimodal.  However, it is certainly of interest in its own right.

For example, in the case of the class of partitions $\lambda = (m,k,1)$, we can state 

\begin{theorem} $f_{(m,k,1),i}(q) = q^{k+i^2-2i+2} \frac{(1-q^{m-k+1})(1-q^{i-1})}{(1-q^i)(1-q)}\left[{k \atop {i-1}} \right]_q \left[ {{m+1} \atop {i-1}} \right]_q.$
\end{theorem}

\begin{proof} We again work by recurrence.  The minimum possible number of descents is $i=2$, and for this the theorem's claim is that $$f_{(m,k,1),2}(q) = q^{k+2} \frac{(1-q^{m-k+1})}{(1-q^2)}\left[{k \atop {1}} \right]_q \left[ {{m+1} \atop {1}} \right]_q.$$

One notices that this is precisely $$f_{(m,k,1),2}(q) = q^{-1} f_{(m+1,k+1),2}(q).$$

Thus, we can prove the base case by establishing a bijection between tableaux of shape $(m,k,1)$ with 2 descents and tableaux of shape $(m+1,k+1)$ with 2 descents that increases major index by 1.  This bijection can be described as follows.

There are two possible forms of such a tableaux of shape $(m,k,1)$: either both descent tops are in the first row, or one each are in the first row and the second row.  We illustrate the two possibilities below.  In these diagrams, $i_1$ and $i_2$ are the descent tops; the second row contains a maximal initial segment of consecutive integers from $i_1 + 1$ to $i_1 + j$, where in the latter case $i_2 = i_1 + j$.  Ellipses are runs of consecutive integers uniquely defined by filling out the necessary shape without any other descents.

\begin{center}
\begin{tabular}{cc}
\begin{tikzpicture}
\draw (0.1,-0.25) node {1};
\draw (1.0,-0.25) node {$\cdots$};
\draw (2.6,-0.25) node {$i_1$};
\draw (3.4,-0.25) node {$\cdots$};
\draw (3.9,-0.25) node {$i_2$};
\draw (4.4,-0.25) node {$\cdots$};
\draw (0.1,-0.75) node {$i_1+1$};
\draw (1.0,-0.75) node {$\cdots$};
\draw (1.8,-0.75) node {$i_1+j$};
\draw (2.6,-0.75) node {$\cdots$};
\draw (0.1,-1.25) node {$i_2+1$};
\draw (-0.4,0) -- (4.7,0) -- (4.7,-0.5) -- (-0.4,-0.5) -- (-0.4,-1) -- (2.9,-1) -- (2.9,0);
\draw (-0.4,0) -- (-0.4,-0.5);
\draw (0.7,0) -- (0.7,-1.5) -- (-0.4,-1.5) -- (-0.4,-1);
\draw (1.2,0) -- (1.2,-1);
\draw (2.3,0) -- (2.3,-1);
\draw (3.7,0) -- (3.7,-0.5);
\draw (4.1,0) -- (4.1,-0.5);
\end{tikzpicture}
&
\begin{tikzpicture}
\draw (0.1,-0.25) node {1};
\draw (1.0,-0.25) node {$\cdots$};
\draw (3.3,-0.25) node {$i_1$};
\draw (4.0,-0.25) node {$\cdots$};
\draw (0.1,-0.75) node {$i_1+1$};
\draw (1.0,-0.75) node {$\cdots$};
\draw (2.1,-0.75) node {$i_1+j=i_2$};
\draw (3.3,-0.75) node {$\cdots$};
\draw (0.1,-1.25) node {$i_2+1$};
\draw (0.7,0) -- (0.7,-1.5) -- (-0.4,-1.5) -- (-0.4,0) -- (4.3,0) -- (4.3,-0.5) -- (-0.4,-0.5);
\draw (-0.4,-1) -- (3.6,-1) -- (3.6,0);
\draw (3.0,0) -- (3.0,-1);
\draw (1.25,-0.5) -- (1.25,-1);
\end{tikzpicture}
\end{tabular}
\end{center}

The bijection is as follows:

\begin{enumerate}
\item Move $i_2+1$ from the third row to the second row immediately after $i_1 + j$.
\item Move $i_1+j$ from the second row to the first row immediately after $i_1$.
\item Increase all entries from $i_1+j$ to $m+k+1$ by 1.
\item Insert a new $i_1+j$ at the end of the maximal streak in the second row before the second descent bottom (possibly empty, if $j = 1$ and the tableaux was of the second type).
\end{enumerate}

To reverse the bijection, move the second descent bottom to a new third row.  Remove $i_1+j$ and reduce $i_1+j+1$ and all higher entries by 1.  Move the new $i_1+j$ back to the second row.  This will cause the second row to contain a descent top if $i_2$ had immediately followed $i_1$ in the first row.

This establishes the theorem for the case $i=2$.

We observe the recurrence

\begin{multline*}f_{(m,k,1),i}(q) = f_{(m-1,k,1),i} + q^{m+k} f_{(m,k),i-1} \\ + \sum_{k_0 = 1}^{k-1} q^{m+k_0+1} f_{(m-1,k_0,1),i-1}  + \sum_{k_0 = 1}^{k-1} q^{m+k_0} f_{(m,k_0),i-1}.
\end{multline*}

The four terms arise thus:

\begin{enumerate}
\item $m+k+1$ is the last element of the first row.  Removal gives the first term.
\item $m+k+1$ is the element of the third row. Removal gives the second term.
\item $m+k+1$ is the last element of the second row, and ends a segment of consecutive integers which terminates at a descent bottom.  Removing the segment and the descent top (the last element of the first row) gives the third term.
\item $m+k+1$ is the last element of the second row, and ends a segment of consecutive integers which follows the element in the third row.  Removing the segment and the third row gives the fourth term.
\end{enumerate}

We now let $i \geq 3$ and assume inductively that the theorem has been proved for smaller $i$, and for the same $i$ with smaller $m$ if $m \neq k$ (the first term of the recurrence yields 0 if $m=k$).  Substitute the appropriate formulae and sum.  We need only the previously stated $q$-binomial summation identities.  We obtain a claimed identity of $q$-binomial coefficients that becomes a claimed finite polynomial identity when multiplied through by

$$ \frac{(q)_i (q)_{k-i+1} (q)_{i-1} (q)_{n-i+2} (1-q) (1-q^{i-1})(1-q^i)}{q^{k+i^2-4i+6}(q)_{k-1}(q)_n}.$$

A short calculation, made easier by a symbolic computation package, verifies the identity and the theorem is proved.
\end{proof}

For general partitions into three parts, a more general formula for two-rowed skew partitions becomes useful. The author is presently pursuing this line of investigation.  However, the recurrence method is limited in its scope of application.  To begin with, a conjectured form is necessary, and for partitions $(m,k,c)$ with $c>1$, the generating function is \emph{not} always a product of this type. In those cases it is likely a sum over binomial coefficients which would have to be determined. In \cite{EC2} the $q$-hooklength formula above is proved by appeal to the theory of symmetric functions; this may be a fruitful route to conjecture and prove a general form.

\section{Acknowledgements}

A portion of this material was presented at the CANT 2017 conference at the CUNY Graduate Center in May 2017.  The author thanks organizer Melvyn Nathanson for the opportunity to speak and the interesting sessions.

\end{document}